\documentclass[12pt,oneside]{article}
\usepackage{amsmath,amssymb,amsfonts,amsthm}
\usepackage{mathrsfs}
\usepackage{eurosym}
\usepackage{geometry}
\usepackage{mathtools}
\usepackage{pdfpages}
\renewcommand{\arraystretch}{1.5}
\usepackage{indentfirst}
\usepackage{hyperref}
\usepackage{cleveref}
\crefformat{section}{\S#2#1#3} 
\crefformat{subsection}{\S#2#1#3}
\crefformat{subsubsection}{\S#2#1#3}
\hypersetup{
	colorlinks=true,
	linkcolor=blue,
	filecolor=magenta,
	urlcolor=cyan,
	citecolor=red}

\newtheorem{theorem}{Theorem}[section]
\newtheorem{remark}[theorem]{Remark}
\newtheorem{lemma}[theorem]{Lemma}
\newtheorem{corollary}[theorem]{Corollary}
\newtheorem{definition}[theorem]{Definition}

\newtheorem{proposition}[theorem]{Proposition}
\numberwithin{equation}{section}

\begin{document}
\title{\bf{ Special anisotropic conformal changes of conic pseudo-Finsler surfaces}}

\author{{\bf Nabil L. Youssef$\,^1$,  Ebtsam H. Taha$^1$,   A. A. Kotb$^2$ and S. G. Elgendi$^{3}$}}
\date{}

\maketitle
\vspace{-1.15cm}
\begin{center}
{$^1$ Department of Mathematics, Faculty of Science,\\
Cairo University, Giza, Egypt}
\vspace{-8pt}
\end{center}

\begin{center}
{$^2$Department of Mathematics, Faculty of Science,\\
Damietta University, Damietta , Egypt}
\vspace{-8pt}
\end{center} 

\begin{center}
 $^{3}$ Department of Mathematics, Faculty of Science,\\ Islamic University of Madinah,  Madinah, Kingdom of Saudi Arabia
\vspace{0.2cm}
\end{center}

\begin{center}
nlyoussef@sci.cu.edu.eg, nlyoussef2003@yahoo.fr\\
ebtsam.taha@sci.cu.edu.eg, ebtsam.h.taha@hotmail.com\\
alikotb@du.edu.eg,  kotbmath2010@yahoo.com\\
salah.ali@fsc.bu.edu.eg, salahelgendi@yahoo.com

\end{center}

\bigskip

\begin{center}
\textit{Dedicated to the memory of Professor Nabil L. Youssef
}\end{center}

\bigskip

\begin{abstract}
 This study presents many special anisotropic conformal changes of a conic pseudo-Finsler surface $(M,F)$, such as $C$-anisotropic and horizontal $C$-anisotropic conformal transformations, which reduce to $C$-conformal when the conformal factor is solely position-dependent.  Furthermore, we present vertical $C$-anisotropic conformal changes and demonstrate that they are characterized by the property of $(M,F)$ being Riemannian.  Additionally, we examine the anisotropic conformal transformation that fulfils the $\phi T$-condition, the horizontal $\phi T$-condition, and the vertical $\phi T$-condition.  The first two conditions reduce to the $\boldsymbol{\sigma} T$-condition when the conformal factor relies solely on a positional variable.  We demonstrate that, under the vertical $\phi T$-condition change, every Landsberg surface is Berwaldian.  Thus, the vertical $\phi T$-condition is equivalent to the $T$-condition.  Furthermore, we examine the scenario when the anisotropic conformal factor becomes the main scalar of the non-Riemannian surface $(M,F)$.  We present an example of a Finslerian Schwarzschild-de Sitter solution having Finslerian spherical symmetry and apply our results to it.
\end{abstract}


\section{Introduction}
A conformal transformation entails the scaling of a Finsler metric by a smooth positive function known as the conformal factor.   A transformation that relies solely on the manifold coordinates (position) is referred to as an isotropic conformal transformation \cite{elgendi, Hashiguchi76, Knebelman, Ichijyo-Hashiguchi, Bin Shen, Shenbook16}.   Knebelman initially introduced the theory of isotropic conformal transformation of Finsler spaces in 1929 \cite{Knebelman}.  M. Hashiguchi \cite{Hashiguchi76} subsequently introduced a particular transformation termed a $C$-conformal change, which serves as an essential tool in the study of Finsler geometry.  This transformation facilitates the examination of invariant qualities and the classification of special  Finsler spaces.  It is an extension of the concept of concurrent vector fields.  Concurrent vector fields have been extensively studied in the Riemannian geometry.  In 1950, Tachibana \cite{Tachibana} expanded this concept to Finsler geometry and identified the spaces in which such vector fields are exist.  In 2019, N. Youssef et al. presented a general idea known as a semi-concurrent vector field in \cite{Youssef semiconcurrent}.  The Cartan tensor is a fundamental non-Riemannian quantity in Finsler geometry, invariant under isotropic conformal changes, encouraging some geometers to define and explore the $C$-conformal transformation.   In Finsler surfaces \cite{An.In.Ma1993, Bacs-Matsumoto, Berwald41, Matsumoto 2003, YANG-CHENG}, the $C$-conformal transformation is characterized by $F$ being a Riemannian metric.

The $T$-tensor is a crucial object in Finsler geometry which holds significant relevance. The $\boldsymbol{\sigma} T$-condition is particularly important in determining the preservation of features such as the Landsberg or Berwald character under conformal transformations.  Elgendi, in \cite{elgendi}, examined the relationship between the $T$-condition and the $\boldsymbol{\sigma} T$-condition, demonstrating their equivalence in Finsler surfaces.   Furthermore, the Landsberg unicorn's problem is closely related to both $T$-condition and $\boldsymbol{\sigma} T$-condition \cite{elgendi,Elgendi-LBp,Elgendi-solutions,Elgendi-ST_condition}.

On the other hand, when the conformal factor is dependent upon both direction (tangent or velocity vectors) and position, the transformation is referred to as an anisotropic conformal transformation \cite{anisotropic_2014, anisotropic, first_paper, second_paper}.
 Anisotropic conformal transformations of Finsler spaces vary in direction, as opposed to uniform scaling in conformal transformations.   The direction-dependent provides more comprehensive metric deformations, essential for describing directional phenomena such as relativistic spacetime.  The investigation of anisotropic transformations enhances the examination of Finsler spaces.

\bigskip

We introduce and investigate the concepts of $C$-anisotropic and $\overline{C}$-anisotropic conformal changes of a pseudo-Finsler surface in \S \ref{section_$C$-conformal}  due to the fact that the Cartan tensor does not remain invariant under an anisotropic conformal transformation of a conic pseudo-Finsler surface \cite{first_paper}.  In particular, we explore under what conditions the anisotropic conformal transformation of a conic pseudo-Finsler surface is $C$-anisotropic and $\overline{C}$-anisotropic conformal change, and identify necessary and sufficient conditions such that the property of $C$-anisotropic is preserved.   When this property is preserved, we determine the conditions under which $F$ is projectively flat.  \\
 Based on the dependence of the anisotropic conformal factor on position and directional arguments, we introduce and examine the horizontal $C$-anisotropic, horizontal $\overline{C}$-anisotropic, vertical $C$-anisotropic and vertical $\overline{C}$-anisotropic conformal changes.    Also, we demonstrated that the conformal transformation between two non-Riemannian metrics is horizontal $C$-anisotropic and horizontal $\overline{C}$-anisotropic if and only if the geodesic spray is invariant (Theorem \ref{theorem-horizotal-$C$-anisotropic}).  We show that the conformal changes of $C$, $\overline{C}$, horizontal $C$, and horizontal $\overline{C}$- anisotropic are the same when the conformal factor is just a function of $x$ only, but the vertical $C$-anisotropic change is not.  Additionally,   the vertical $C$-anisotropic conformal change is characterized by the property that the Finsler surface is Riemannian. \\
 
\bigskip
 
We are motivated to introduce the $\phi T$-condition, horizontal $\phi T$-condition and vertical $\phi T$-condition in light of the fact that the $T$-tensor is not invariant under an anisotropic conformal transformation of a conic pseudo-Finsler surface (with conformal factor $\phi (x,y)$) \cite{first_paper, second_paper}.  The anisotropic $\phi T$-, horizontal $\phi T$-, and vertical $\phi T$-conditions are investigated and characterized in ~\S \ref{section_phi-$T$-condition}.  Given an anisotropic conformal transformation that satisfies the anisotropic $\phi T$-condition in $(M,F)$ and $(M,\overline{F})$, we determine the conditions under which the Landsberg metric is Berwaldian, as shown in Theorem \ref{in unicorn 1, 2}.  In ~\S \ref{special conformal factor}, we assume that $F$ represents a non-Riemannian metric, with the anisotropic conformal factor being the main scalar $\mathcal{I}$ of $F$.  This specific case for the conformal factor yields interesting results. \\
In \S \ref{Sec_6},  an example of a Finslerian sphere is provided by an anisotropic conformal transformation of a Riemannian sphere.  We examine the conditions under which this transformation meets the criteria for $C$-anisotropic, $\overline{C}$-anisotropic, horizontal $C$-anisotropic and the $\phi T$-condition.  This example represents a Finslerian Schwarzschild-de Sitter solution that possesses the symmetry of a Finslerian sphere.  This is because, the Finslerian analogue of Birkhoff’s theorem states that a Finslerian gravitational field with the symmetry of a "Finslerian sphere" in a vacuum must be static \cite{PRD2014}.  Also, the Einstein field equations including a non-zero cosmological constant have been modified to the Finslerian framework as presented in \cite{Finsler gravity2024}. This modification yields a Finslerian Schwarzschild-de Sitter solution that possesses the symmetry of a Finslerian sphere.   Finally, we end our work by some concluding remarks along with a table which summarizes all special anisotropic conformal transformations we studied in \S \ref{Concluding Remarks}.
  
\section{Notation and preliminaries}

Let $ M $ be an n-dimensional smooth manifold. The tangent bundle of $M$
is denoted by $(TM,\pi_{M},M )$, where $TM$ is the disjoint union of all tangent vectors at each point in base manifold $M$   and $ \pi_{M} : TM \longrightarrow M $ denote the canonical projection from  $TM$ onto $M$. Define $TM_0=TM\setminus(0)$ as the tangent bundle with the zero section removed. Then $(TM_0,\pi_{M},M  )$ is sub-bundle of non zero vectors. On the base manifold $M$, we use local coordinates $(x^{i})$. In the tangent bundle $TM$, the coordinates are given by $(x^{i}, y^i)$ 
, where $x^{i}$ are the base coordinates and $ y^i$ represent the components of the tangent vectors. 

A function  $f \in C^{\infty} (TM_0)$ is said to be  positively homogeneous of degree $r$ in  the directional  argument $y$ and  denoted by $h(r)$ if $f(x,\lambda y)=\lambda^r f(x,y), \quad \forall \lambda>0.$

A conic sub-bundle of $TM$ is a non-empty open subset $\mathcal{A} $ of $TM_0$ which satisfies $ \pi(\mathcal{A}) = M$ and is  invariant under scaling of tangent vectors by positive real numbers. 

\begin{definition}\label{pseudo-Finsler def}
  A conic pseudo-Finsler  metric on $ M $ is  a smooth function
  $\; F:\mathcal{A} \longrightarrow \mathbb{R}$
  which satisfies the following conditions:
  \begin{description}
    \item[(i)] $ F(x,y)$ is $h(1)$: $F(x,\lambda y)=\lambda F(x,y)$  for all  $\lambda \in \mathbb{R^+},$
    \item[(ii)] For each point of $ \mathcal{A}$, the Finsler metric tensor $ g_{ij} (x,y)=\frac{1}{2}\dot{\partial}_{i}\dot{\partial}_{j}F^{2} (x,y),$
    is non-degenerate.
  \end{description}
The pair $(M, F)$ is called a conic pseudo-Finsler manifold.
\end{definition}

For a Finsler metric $F$, a unique nonlinear Cartan (Ehresmann) connection exists on $\mathcal{A} \subset TM$ that is both torsion-free and compatible with $F$. The nonlinear connection coefficients are determined by 
$$ G_i^j = \frac{1}{4} \dot{\partial}_i \left[ g^{jk} \left( y^m \partial_m \dot{\partial}_k F^2 - \partial_k F^2 \right) \right]. $$
This nonlinear connection defines the horizontal derivatives $\delta_i:=\partial_{i}-G^j_i\dot{\partial}_j$, where $\dot{\partial}_i:=\dfrac{\partial}{\partial y^i}$. \\
 The coefficients of the geodesic spray coefficients of $F$ can be expressed as $$ G^i = \frac{1}{4}  g^{ik} \left( y^m \partial_m \dot{\partial}_k F^2 - \partial_k F^2 \right) .$$
It is clear that  $G^i$ are  smooth $h(2)$ functions  in $\mathcal{A}$, moreover,  the geodesic spray can be defined globally on $TM_{0}$ by
 $ S = y^i \partial_i - 2 G^i  \dot{\partial _i}.$

\bigskip

In the case of positive definite of Finsler surface, Berwald introduced a global frame in \cite{Berwald41} called Berwald frame.
Given that the angular metric $h_{ij}$  has a matrix representation $(h_{ij})$  with rank one, we can determine a unique vector field 
 $m_i(x,y)$ and sign $\varepsilon=\pm1$, where 
\begin{equation*}
h_{ij}=\varepsilon m_im_j.
\end{equation*}
As, $g_{ij}=\ell_{i}\ell_{j}+h_{ij}$, where $\ell_i=\dot{\partial_{i}}F$ and $\ell^{i}=\frac{y^i}{F}$. Then, the metric tensor can be expressed as  
\begin{equation}
g_{ij}=\ell_{i}\ell_{j}+\varepsilon m_im_j .
\end{equation}
We denote $\mathfrak{g}=\det(g_{ij})$ and we have the following relations:
\begin{equation}\label{vanishing li mi}
\ell^{i}\ell_{i}=1\qquad \ell^{i}m_{i}=\ell_{i}m^{i}=0,\qquad m^{i}m_{i}=\varepsilon.
\end{equation} 
Hence, $(\ell^i,m^i)$ is orthonormal frame, called modified Berwald frame introduced by Báscó and Matsumoto \cite{Bacs-Matsumoto} which is suitable for both positive definite and non-positive definite surfaces.\\

The main scalar $\mathcal{I}(x,y)$ is one of the most significant quantities in Finsler surfaces, it is an $h(0)$  smooth function defined from the Cartan tensor \cite{Berwald41} by,
\begin{equation}\label{cartan tensor Finsler surface} 
 F C_{ijk}=\mathcal{I}\;m_{i}m_{j}m_{k}.
\end{equation}

For a smooth $f\in C^{\infty}(TM_0)$ we define the v-scalar derivatives $(f_{;1}, f_{;2})$ and 
h-scalar derivatives $(f_{,1}, f_{,2})$ in $(M,F)$ as follows:
\begin{equation}\label{vertical and horizontal deri of f}
F\dot{\partial}_i f=f_{;1}\ell_{i}+f_{;2}m_{i}, \qquad \delta_i f=f_{,1}\ell_{i}+f_{,2}m_{i},
\end{equation}
where
\begin{align*}
f_{;1}=y^{i}\dot{\partial}_i f,\quad\qquad f_{;2}=\varepsilon F(\dot{\partial}_{i}f)m^{i},
\qquad f_{,1}=(\delta_{i}f)\ell^{i},\quad\qquad f_{,2}=\varepsilon (\delta_{i}f)m^{i}.
\end{align*}
Specifically, if $ f$ is $ h(r) $, then  $f_{;1}=rf$, the commutation formulas \cite{Berwald41} are given by:
\begin{align}
 f_{,1,2}-f_{,2,1}=&-Rf_{;2},\label{first comutation}\\
 f_{,1;2}-f_{;2,1}=&f_{,2},\label{second comutation}\\
f_{,2;2}-f_{;2,2}=&-\varepsilon(f_{,1}+\mathcal{I}f_{,2}+\mathcal{I}_{,1}f_{;2}),\label{third comutation}
\end{align}
where R is said to be the Gauss curvature or the h-scalar curvature.

\begin{definition}\emph{\cite{first_paper}}
The anisotropic conformal change of  a conic pseudo-Finsler metric $F $ is defined by
\begin{equation}\label{the anisotropic conformal transformation}
 F\longmapsto \overline{F}(x,y)=e^{\phi(x,y)}F(x,y), \quad F^{2}(\dot{\partial}_{i}\dot{\partial}_{j}\phi\
 +(\dot{\partial}_{i}\phi)(\dot{\partial}_{j}\phi)) m^{i}m^{j}+\varepsilon=\varepsilon+\sigma-(\phi_{;2})^2\neq0,
 \end{equation}
  \begin{align}
  \sigma&=\phi_{;2;2}+\varepsilon \mathcal{I}\phi_{;2}+2\,(\phi_{;2})^{2},\label{formula of sigma}
\end{align}
    given that the conformal factor $ \phi(x,y)$ is a smooth $h(0)$-function on $\mathcal{A}$.
In this case,  $F$ is said to be anisotropic conformally changed to $\overline{F}$. In addition, we say that the anisotropic conformal transformation is proper if the conformal factor is a non-isotropic and non-homothetic function which means $\phi_{;2} \neq 0$.
\end{definition}

Now, we define the v-scalar derivatives $(f_{;\,a}, f_{;\,b})$ and h-scalar derivatives $(f_{,\,a}, f_{,\,b})$ in $(M,\overline{F})$ for  $f$ by:
\begin{align*}
\overline{F}\dot{\partial}_i f=f_{;\,a}\overline{\ell}_{i}+f_{;\,b}\overline{m}_{i}, \qquad \overline{\delta}_i f=f_{,\,a}\overline{\ell}_{i}+f_{,\,b}\overline{m}_{i},
\end{align*}
 where
\begin{align*}
f_{;\,a}=y^{i}\dot{\partial}_i f,\qquad \quad f_{;\,b}=\varepsilon \overline{F}(\dot{\partial}_{i}f)\overline{m}^{i},\qquad 
f_{,\,a}=(\overline{\delta}_{i}f)\overline{\ell}^{i}, \qquad \quad f_{,\,b}=\varepsilon (\overline{\delta}_{i}f)\overline{m}^{i}.
\end{align*}
In \cite{first_paper}, we discuss the anisotropic conformal change on a conic pseudo-Finsler surface $F$ equipped with modified Berwald frame and determined how this change affects the components of the Berwald frame, the main scalar and the geodesic spray coefficients of $\overline{F}$ as follows:
\begin{align}\label{transform of Berwald frame}
  \overline{\ell}_{i}=e^{\phi}[\ell_{i}+\phi_{;2}\; m_{i}],\qquad
   \overline{\ell}^{i}=e^{-\phi}\ell^{i},\qquad
  \overline{m}_{i}=e^{\phi}\sqrt{\frac{\varepsilon}{\rho}}m_{i},\qquad
  \overline{m}^{i}=e^{-\phi}\sqrt{\varepsilon\rho}[m^{i}-\varepsilon \phi_{;2}\;\ell^{i}].
  \end{align}
  \begin{align} 
  \overline{\mathcal{I}}&=\sqrt{\varepsilon\rho}[\mathcal{I}+2\varepsilon\phi_{;2}-\frac{\varepsilon\rho_{;2}}{2\rho}].\label{transform of main scalar 2}\\
  \overline{G}^{i}&=G^i+Qm^i+P\ell^i,\label{transform of coefficient spray}
  \end{align}
where 
  \begin{align}
\rho&=\frac{1}{\sigma+\varepsilon-(\phi_{;2})^{2}}.\label{formula of rho}\\
2Q&=\varepsilon\rho F^2(\phi_{;2}\phi_{,1}+\phi_{,1;2}-2\phi_{,2}),\label{formula of Q only}\\
2P&=-\rho F^2\phi_{;2}(\phi_{;2}\phi_{,1}+\phi_{,1;2}-2\phi_{,2})+F^2\phi_{,1}.\label{formula of P only}
\end{align}
Furthermore, from \eqref{formula of Q only} and \eqref{formula of P only}, we get
\begin{equation}\label{relation between P and Q}
2\varepsilon \phi_{;2}Q+2P=F^2\phi_{,1}.
\end{equation}

\begin{lemma}\label{lemma of condition to homthety}
    Let $(M,F)$ be Finsler surface. Each smooth  function  $f$ is  on $M$ satisfies $m^i\partial_i f=0$ is  constant.
\end{lemma}
\begin{proof}
  Since $f \in C^{\infty}(M)$, then we have $\delta_{i}f =\partial_{i}f \text{ and } \delta_{i}f =f_{,1}\ell_{i}+f_{,2}m_{i}.$ Then, $m^i\partial_i f=\varepsilon f_{,2}$. Moreover, if $m^i\partial_i f=0$, we obtain $f_{;1}=f_{;2}=f_{,2}=0$.  Thereby, using the commutation formula \eqref{third comutation}, we get $f_{,1} =0$. Consequently, $f$ is a constant function.
\end{proof}

\section{$C$-anisotropic and horizontal $C$-anisotropic change }\label{section_$C$-conformal}
Hashiguchi \cite{Hashiguchi76} introduced the concept of $C$-conformal transformation in Finsler geometry as a special form of a conformal transformation. This class is a subset of conformal transformations, distinguished by specific conditions on the Cartan tensor and the conformal factor. In Finsler surfaces, the $C$-conformal transformation is characterized by the fact that $F$ is a Riemannian metric. More precisely,
\begin{definition} \emph{\cite{Hashiguchi76}}
The (isotropic) conformal change $\overline{F}(x,y)=e^{\theta (x)}F(x,y)$ is said to be $C$-conformal if it is not a homothetic change and satisfies
\begin{align*}\label{def of $C$-conformal}
C^i_{jk}\partial_i\theta=0.
\end{align*}
\end{definition}

\begin{lemma}
The conformal change of a two-dimensional conic pseudo-Finsler metric $F$ defined as $\overline{F}=e^{\theta }F$ is $C$-anisotropic conformal if and only if $F$ is Riemannian.
\end{lemma}
\begin{proof}
Let  $\overline{F}=e^{\theta}F$ be a $C$-conformal transformation, then $$0=C^i_{jk}\partial_{i}\theta=(\partial_{i}\theta)\frac{\mathcal{I}}{F}m^im_jm_k,$$
which is equivalent to either $\mathcal{I}=0$ ($F$ is Riemannian) or $m^i\partial_i\theta=0.$
Since $\theta$ is a function of $x$ only, then, by Lemma \ref{lemma of condition to homthety}, $\theta$ is  constant which is a contradiction. 
\end{proof}

  It should be noted that, in the case of a conformal transformation, $\overline{C}^{i}_{jk}=C^{i}_{jk}$ as the conformal factor depends only on the position only.  However, in the case of the anisotropic conformal transformation \eqref{the anisotropic conformal transformation}, we have  $\overline{C}^{i}_{jk} \neq C^{i}_{jk}$ as the conformal factor depends on both $x$ and $y$. This inspires us to introduce some special anisotropic conformal transformations, namely, $C$-anisotropic conformal and horizontal $C$-anisotropic conformal transformations. We also discuss under what conditions these transformations are preserved.
 \begin{definition}
 The proper anisotropic conformal change $\overline{F}=e^{\phi}F$ is said to be $C$-anisotropic conformal if it satisfies
$C^i_{jk}\partial_i\phi=0.$  Similarly, the anisotropic conformal change $\overline{F}=e^{\phi}F$ is said to be $\overline{C}$-anisotropic conformal   if it satisfies
$\overline{C}^i_{jk}\partial_i\phi=0.$
\end{definition}

\begin{proposition}\label{proposition1 of bar$FC$-conformal}
Let $F$ be conformally anisotropic changed to $\overline{F}$ by \eqref{the anisotropic conformal transformation}. Then, we have:
\begin{description}
    \item[(i)] The anisotropic conformal transformation is $C$-anisotropic conformal if and only if either $F$ is a Riemannian metric or $m^i\partial_i\phi=0$.
    \item[(ii)] The anisotropic conformal transformation is $\overline{C}$-anisotropic conformal if and only if either $\mathcal{I}=\frac{\varepsilon}{2}(\ln{\rho}-4\phi)_{;2}$ ($\overline{F}$ is a Riemannian metric) or $m^i\partial_i\phi-\varepsilon\phi_{;2}\ell^i\partial_i\phi=0$. 
    \item[(iii)] Assume that \eqref{the anisotropic conformal transformation} is an anisotropic conformal transformation between two non-Riemannian metrics. The property of $C$-anisotropic conformal is invariant if and only if $\ell^i\partial_i\phi=0.$
\end{description} 
\end{proposition}
\begin{proof}
\begin{description}
    \item[(i)]The anisotropic conformal transformation is $C$-anisotropic conformal if and only if 
\begin{equation*}
0=C^i_{jk}\partial_{i}\phi=\mathcal{I}m^im_jm_k \partial_{i}\phi.
\end{equation*}
From which we have $\mathcal{I}m^i\partial_{i}\phi=0$. That is,  the anisotropic conformal transformation \eqref{the anisotropic conformal transformation} is $C$-anisotropic conformal if and only if either $F$ is a Riemannian metric or $m^i\partial_i\phi=0$.
\item[(ii)] The anisotropic conformal transformation is $\overline{C}$-anisotropic conformal   if and only if 
\begin{equation*}
0=\overline{C}^i_{jk}\partial_{i}\phi=\mathcal{\overline{I}}(m^im_jm_k-\varepsilon\phi_{;2}\ell^im_jm_k) \partial_{i}\phi.
\end{equation*}
Therefore, the anisotropic conformal transformation \eqref{the anisotropic conformal transformation} is $\overline{C}$-anisotropic conformal   if and only if either $\mathcal{\overline{I}}=0$ or $m^i\partial_i\phi-\varepsilon\phi_{;2}\ell^i\partial_i\phi=0.$ By making use of \eqref{transform of main scalar 2},   \eqref{the anisotropic conformal transformation} is $\overline{C}$-anisotropic conformal   if and only if either $\mathcal{I}=\frac{\varepsilon}{2}[\frac{\rho_{;2}}{\rho}-4\phi_{;2}]=\frac{\varepsilon}{2}(\ln{\rho}-4\phi)_{;2}$ or $m^i\partial_i\phi-\varepsilon\phi_{;2}\ell^i\partial_i\phi=0.$ 
\item[(iii)] Consider the anisotropic conformal change \eqref{the anisotropic conformal transformation} between  non-Riemannian metrics $F$ and $\overline{F}$. Now, we have:

$(\Longrightarrow)$ If the property of $C$-anisotropic conformal is invariant, that is, our anisotropic conformal transformation is both $C$-anisotropic and $\overline{C}$-anisotropic transformation, then we have $m^i\partial_i\phi=0$ and $m^i\partial_i\phi-\varepsilon\phi_{;2}\ell^i\partial_i\phi=0$. Since $\phi_{;2}\neq0$, we conclude  $\ell^i\partial_i\phi=0.$

$(\Longleftarrow)$ If $\ell^i\partial_i\phi=0$ and the anisotropic conformal change \eqref{the anisotropic conformal transformation} is $C$-anisotropic change (i.e., $m^i\partial_i\phi=0$, by \textbf{(i)} above), then, since $\phi_{;2}\neq0$, we  have $m^i\partial_i\phi-\varepsilon\phi_{;2}\ell^i\partial_i\phi=0$. That is, \eqref{the anisotropic conformal transformation} is $\overline{C}$-anisotropic conformal change.
\end{description}
\end{proof}
     
\begin{remark}\label{projective surfaces}
    A Finsler surface $(M,F)$ is projectively flat if and only if it satisfies the Hamel’s equation $(\dot{\partial}_1\partial_{2}F=\dot{\partial}_2\partial_{1}F)$ \emph{\cite[Eqn. (2.70)]{Guo-Mo-book}}. The geodesic coefficient of the two-dimensional metric $F$ is written as 
$2G^i=y^r(\partial_r F)\ell^i+\frac{F^2(\dot{\partial_2}\partial_1 F-\dot{\partial_1}\partial_2 F)}{h}m^i,$ where $h=\frac{\varepsilon}{\ell^1m^2-\ell^1m^1}$ \emph{ \cite[Eqn. (1.4)]{Matsumoto 2003}.} Consequently, the Finsler surface $(M,F)$ is projectively flat if and only if $G^km_k=0.$
\end{remark}
\begin{proposition}
     Assuming that $\overline{F}=e^{\phi}F$ is an anisotropic conformal transformation between two non-Riemannian metrics, we have the following:
     \begin{description}
         \item[(i)]If $F$ is a locally Minkowski metric and the anisotropic conformal transformation \eqref{the anisotropic conformal transformation} is $\overline{C}$-anisotropic conformal change, then the two sprays $S$ and $\overline{S}$ are projectively equivalent if and only if $\phi_{;2,1}=0$ in some coordinate system.
         \item[(ii)] If $F$ is projectively flat metric, then the anisotropic conformal transformation \eqref{the anisotropic conformal transformation} is $\overline{C}$-anisotropic if and only if $\phi_{,2} = \phi_{;2}\left[\phi_{,1} - \frac{G^k}{F^2}\ell_k\right]$. Moreover,  the anisotropic conformal transformation \eqref{the anisotropic conformal transformation} is $C$-anisotropic if and only if $\phi_{,2}=-\frac{\phi_{;2}}{F^2}G^km_k. $
         \item[(iii)] If the property of $C$-anisotropic conformal is preserved, then $F$ is projectively flat  if and only if  $\phi$ is a first integral of the geodesic spray $S.$
     \end{description}
\end{proposition}
\begin{proof}
For a Finsler surface equipped with modified Berwald frame, we have, by \cite[Eqn. (4.8)]{first_paper}, the following equalities:
\begin{equation}
\renewcommand{\arraystretch}{1.5}
  \left.\begin{array}{r@{\;}l}
{\textbf{(a)}}& F^2\ell^k\,\partial_k\phi=F^2\phi_{,1}+2G^k\phi_{;2}\;m_k\\
{\textbf{(b)}}& F m^k\,\partial_k\phi=\varepsilon F\,\phi_{,2}+G^i_{k}\;\phi_{;2}\;m^k m_i\\
  \end{array}\right\} \label{eq}
\end{equation}
This leads to the equivalence:
\begin{align}
m^i\partial_i\phi-\varepsilon\phi_{;2}\ell^i\partial_i\phi=0&\Longleftrightarrow \phi_{,2}=\phi_{;2}[\phi_{,1}+\frac{2\phi_{;2}}{F^2}G^km_k-\frac{\varepsilon }{F}G^i_km^km_{i}],\label{equivalence 1 of C bar}\\
m^i\,\partial_i\phi=0&\Longleftrightarrow \phi_{,2}=-\frac{\varepsilon\phi_{;2}}{F}G^i_{k}m^k m_i,\label{equivalence 2 of C bar}\\
       \ell^i\partial_i\phi=0&\Longleftrightarrow \phi_{,1}=-\frac{2}{F^2}\phi_{;2}G^k\;m_k.\label{equivalence 3 of C bar}
  \end{align}
\begin{description}
\item[(i)] Let $F$ be a locally Minkowski metric $F$ (i.e.,  $G^i=0$ in some coordinate system \cite[Definition 5.3]{Shenbook16}) and  \eqref{the anisotropic conformal transformation} be $\overline{C}$-anisotropic. By Proposition \ref{proposition1 of bar$FC$-conformal} \textbf{(ii)} and \eqref{equivalence 1 of C bar}, we have
\begin{equation}\label{reduction of Q}
\phi_{,2}=\phi_{;2}\phi_{,1}.
\end{equation}
Using the commutation formula \eqref{second comutation} and \eqref{formula of Q only}, along with \eqref{reduction of Q}, we get 
  \begin{equation*}
      Q=\frac{1}{2}\varepsilon\rho F^2(\phi_{;2}\phi_{,1}+\phi_{,1;2}-2\phi_{,2})=\frac{1}{2}\varepsilon\rho F^2(\phi_{;2}\phi_{,1}+\phi_{;2,1}-\phi_{,2})=\frac{1}{2}\varepsilon\rho F^2\phi_{;2,1}.
  \end{equation*}
  Consequently, the two sprays $S$ and $\overline{S}$ are projectively equivalent if and only if $\phi_{;2,1}=0$ in some coordinate system.
  \item[(ii)] Assuming that $F$ is projectively flat, by Remark \ref{projective surfaces}, we have
  \begin{equation}\label{projectively_flat_F}
      G^km_k=0.
  \end{equation}
  Differentiating \eqref{projectively_flat_F} with respect to $y^i$, where $F\dot{\partial _{j}}m_{i}=-(\ell_{i}-\varepsilon \mathcal{I} m_{i})m_{j}$, we get\\ $0= G^k_im_k+\frac{G^k}{F}(-\ell_{k}+\varepsilon \mathcal{I} m_{k})m_{i}$. Using \eqref{projectively_flat_F}, we obtain 
  \begin{equation}\label{diff_G^km_k}
     G^k_im_km^i=\frac{\varepsilon G^k}{F}\ell_{k}.
  \end{equation}
  Substituting \eqref{projectively_flat_F} and \eqref{diff_G^km_k} into \eqref{equivalence 1 of C bar}, then by Proposition \ref{proposition1 of bar$FC$-conformal} \textbf{(ii)}, the anisotropic conformal transformation \eqref{the anisotropic conformal transformation} is $\overline{C}$-anisotropic conformal   if and only if  $\phi_{,2}=\phi_{;2}\Big{[}\phi_{,1}-\frac{G^k}{F^2}\ell_k\Big{]}.$ On the other hand, plugging  \eqref{diff_G^km_k} into \eqref{equivalence 2 of C bar} leads to the anisotropic conformal transformation \eqref{the anisotropic conformal transformation} is $C$-anisotropic if and only if $\phi_{,2}=-\frac{\phi_{;2}}{F^2}G^k\ell_k. $
  \item[(iii)]  Let the property of $C$-anisotropic conformal be  preserved under the anisotropic conformal change \eqref{the anisotropic conformal transformation}, form Proposition \ref{proposition1 of bar$FC$-conformal} \textbf{(iii)} and \eqref{equivalence 3 of C bar}, we have 
  $\phi_{,1}=-\frac{2}{F^2}\phi_{;2}G^k\;m_k.$
  Since $\phi_{;2}\neq0$, then, by Remark \ref{projective surfaces}, the  metric $F$ is projectively flat ($G^km_k=0$) if and only if the anisotropic conformal factor is a first integral of the geodesic spray $S$ (i.e. $\phi_{,1}=0$). 
  \end{description}
  \vspace*{-0.75 cm}\[\qedhere\]
\end{proof}

\begin{remark} \label{remark of $C$-coformal bar $C$-conformal with isot. cofor. fact.}
\begin{description}
    \item[(i)] Let $\overline{F}=e^{\phi}F$ be non-Riemannian with $\phi$ is horizontally constant (i.e. $\delta_{i}\phi =0$), then the anisotropic conformal transformation \eqref{the anisotropic conformal transformation} is $\overline{C}$-anisotropic if and only if $2\phi_{;2}G^km_k=\varepsilon F G^i_km^km_{i},$ by Proposition \ref{proposition1 of bar$FC$-conformal} \textbf{(ii)} and \eqref{equivalence 1 of C bar}.
    \item[(ii)] When the conformal factor in \eqref{the anisotropic conformal transformation} is a function of positional argument, both $C$-anisotropic and $\overline{C}$-anisotropic conformal transformations reduce to the well-known $C$-conformal transformation.
\end{description} 
\end{remark}

\begin{definition}
The anisotropic conformal change \eqref{the anisotropic conformal transformation} is said to be horizontal $C$-anisotropic conformal if  $(\delta_{i}\phi)C^i_{jk}=0.$ Similarly, it is said to be  horizontal $\overline{C}$-anisotropic conformal if $(\delta_{i}\phi)\overline{C}^i_{jk}=0.$  
\end{definition}

\begin{theorem}\label{theorem-horizotal-$C$-anisotropic}
Let $(M,F)$ be a Finsler surface and \eqref{the anisotropic conformal transformation} be the anisotropic conformal transformation. Then we have the following:
\begin{description} 
\item[(i)]  the anisotropic conformal transformation \eqref{the anisotropic conformal transformation} is horizontal $C$-anisotropic conformal if and only if either $F$ is Riemannian or $\phi_{,2}=0.$
\item[(ii)]  the anisotropic conformal transformation \eqref{the anisotropic conformal transformation} is  horizontal $\overline{C}$-anisotropic conformal if and only if $\mathcal{I}=\frac{\varepsilon}{2}(\ln{\rho}-4\phi)_{;2}$ or $\phi_{,2}=\phi_{,1}\phi_{;2}.$ 
\item[(iii)]  The anisotropic conformal change \eqref{the anisotropic conformal transformation} between non-Riemannian surfaces is a horizontal $C$-anisotropic and $\overline{C}$-anisotropic if and only if the geodesic spray is invariant.
\end{description}
\end{theorem}
\begin{proof}
\begin{description}
    \item[(i)] The anisotropic conformal transformation \eqref{the anisotropic conformal transformation} is horizontal $C$-anisotropic conformal means that
\begin{equation*}
 0=(\delta_{i}\phi)C^i_{jk}=(\phi_{,1}\ell_i+\phi_{,2}m_i)\frac{\mathcal{I}}{F}m^im_jm_k=\frac{\varepsilon\phi_{,2}\mathcal{I}}{F}m_jm_k,
\end{equation*}
which is equivalent that either  $F$ being Riemannian or $\phi_{,2}=0.$
\item[(ii)] From \eqref{transform of Berwald frame} and \eqref{transform of main scalar 2}, the anisotropic conformal transformation \eqref{the anisotropic conformal transformation} is a horizontal $\overline{C}$-anisotropic conformal   if and only if
\begin{align}
0=&(\delta_{i}\phi)\overline{C}^i_{jk}=\frac{1}{F}[(\phi_{,1}\ell_i+\phi_{,2}m_i)(\mathcal{I}+2\varepsilon\phi_{;2}-\frac{\varepsilon\rho_{;2}}{2\rho})(m^im_jm_k-\varepsilon\phi_{;2}\ell^im_jm_k)]\nonumber\\
=&\frac{1}{F}[(\mathcal{I}+2\varepsilon\phi_{;2}-\frac{\varepsilon\rho_{;2}}{2\rho})(\varepsilon \phi_{,2}m_jm_k-\varepsilon\phi_{,1}\phi_{;2}m_jm_k)].\nonumber
\end{align}
Hence, the anisotropic conformal transformation \eqref{the anisotropic conformal transformation} is horizontal $\overline{C}$-anisotropic conformal if and only if $\mathcal{I}=\frac{\varepsilon}{2}[\frac{\rho_{;2}}{\rho}-4\phi_{;2}]=\frac{\varepsilon}{2}(\ln{\rho}-4\phi)_{;2}$  or $\phi_{,2}=\phi_{,1}\phi_{;2}.$
\item[(iii)] 
    ($\Longleftarrow$) If the geodesic spray is invariant, then $\phi$ is horizontally constant \cite[Theorem 4.11]{first_paper}.  Consequently, the anisotropic conformal transformation \eqref{the anisotropic conformal transformation}  is  horizontal $C$-anisotropic and $\overline{C}$-anisotropic.\\
    ($\Longrightarrow$) If the anisotropic conformal change \eqref{the anisotropic conformal transformation} is a horizontal $C$-anisotropic and $\overline{C}$-anisotropic. Also, $F$ and $\overline{F}$ are non-Riemannian. Then, we get 
    $\phi_{,2}=0, \text{ and } \phi_{,2}=\phi_{;2}\phi_{,1}.$ Therefore,  $\phi_{,1}=\phi_{,2}=0$ (as  $\phi_{;2}\neq0$) which implies that  the geodesic spray is invariant.
\end{description}
\vspace*{-0.9 cm}\[\qedhere\]
\end{proof}

\begin{proposition}
    Let $(M,F)$\emph{ (}resp. $(M,\overline{F})$\emph{) }be non-Riemannian surface and \eqref{the anisotropic conformal transformation} be a horizontal $C$-anisotropic \emph{(}resp. horizontal $\overline{C}$-anisotropic\emph{)}. The geodesic spray is invariant if and only if the conformal factor is first integral of $S$ \emph{ (}resp. either $\phi_{,2}=0$ or $\phi_{,1}=0\emph{) }$.
\end{proposition}
\begin{proof}
Let the conformal factor be first integral of $S$ and \eqref{the anisotropic conformal transformation} be  horizontal $C$-anisotropic conformal, provided that $F$ is non-Riemannian, then by
Proposition \ref{theorem-horizotal-$C$-anisotropic}, we get $ \phi_{,1}=0=\phi_{,2}$ which is equivalent to  $\phi$ is constant horizontally which means the geodesic spray is invariant under the anisotropic change \eqref{the anisotropic conformal transformation} by \cite[Theorem 4.11]{first_paper}. In the case of $(M,\overline{F})$ being a non-Riemannian surface and \eqref{the anisotropic conformal transformation} being horizontally $\overline{C}$-anisotropic, we have $\phi_{,2}=\phi_{;2}\phi_{,1}$. Then the geodesic spray is invariant if and only if either $\phi_{,2}=0$ or $\phi_{,1}=0. $
\end{proof}

\begin{proposition}
 Let $(M,\overline{F})$ be non-Riemannian surface and \eqref{the anisotropic conformal transformation} be a horizontal $\overline{C}$-anisotropic. The two geodesic sprays $S$ and $\overline{S}$ are projectively equivalent if and only if $\phi_{;2,1}=0$  \emph{(} i.e. $\phi_{;2}$ is first integral of the geodesic spray $S$\emph{).}
\end{proposition}
\begin{proof}
The two  geodesic sprays $S$ and $\overline{S}$ are projectively equivalent if and only if $Q=0$, by \cite[Theorem 5.2]{first_paper}. Then from \eqref{formula of Q only} and \eqref {second comutation}, we have
    \begin{equation}\label{C bar horz and projec equi}
        0=\phi_{;2}\phi_{,1}+\phi_{,1;2}-2\phi_{,2}=\phi_{;2}\phi_{,1}+\phi_{;2,1}-\phi_{,2}
    \end{equation}
    Since \eqref{the anisotropic conformal transformation} is a horizontal $\overline{C}$-anisotropic conformal change, provided that $(M,\overline{F})$ is a non Riemannian metric, then from Proposition \ref{theorem-horizotal-$C$-anisotropic} \textbf{(ii)} and \eqref{C bar horz and projec equi}, the two geodesic sprays $S$ and $\overline{S}$ are projectively equivalent if and only if $\phi_{;2,1}=0$.
\end{proof}
\begin{remark}
    Let \eqref{the anisotropic conformal transformation} be the anisotropic conformal transformation between non-Riemannian metrics. We have the following:
    \begin{description}
        \item[(i)]From the commutation formula \eqref{second comutation}, the conditions that $\phi$ and $\phi_{;2}$ are first integral of the geodesic spray $S$ is equivalent to that $S$ and $\overline{S}$ are invariant under the anisotropic conformal change \eqref{the anisotropic conformal transformation}. 
        \item[(ii)] The anisotropic conformal transformation \eqref{the anisotropic conformal transformation} is  $C$-anisotropic and horizontal $C$-anisotropic conformal  if and only if $G^i_{k}m^k m_i=0,$ which follows from Proposition \ref{proposition1 of bar$FC$-conformal} \textbf{(i)}, \eqref{equivalence 2 of C bar} and Theorem \ref{theorem-horizotal-$C$-anisotropic}.  
        \item[(iii)] The anisotropic conformal transformation \eqref{the anisotropic conformal transformation} is $\overline{C}$-anisotropic and horizontal $\overline{C}$-anisotropic conformal if and only if $2\phi_{;2} G^km_k=\varepsilon F G^i_km^km_{i}.$ 
    \end{description}
\end{remark}

We have discussed the $C$-anisotropic conformal transformation ($(\partial_{i}\phi)C^i_{jk}=0$) and horizontal $C$-anisotropic conformal transformation ($(\delta_{i}\phi)C^i_{jk}=0$). These two types of anisotropic conformal transformation reduce to $C$-conformal transformation when the conformal factor $\phi$ is a function of $x$ only.  If we take the condition as $(\dot{\partial}_i\phi)C^i_{jk}=0$ (what we call \lq \lq vertical $C$ anisotropic conformal"), then it does not reduce to the $C$-conformal transformation.
\begin{definition}
The proper anisotropic conformal change \eqref{the anisotropic conformal transformation} is said to be  vertical $C$-anisotropic conformal change if  $(\dot{\partial}_i\phi)C^i_{jk}=0$.  Similarly,  \eqref{the anisotropic conformal transformation}  is said to be vertical $\overline{C}$-anisotropic conformal change if $(\dot{\partial}_i\phi)\overline{C}^i_{jk}=0.$
\end{definition}
\begin{proposition}\label{prposition vertical $C$-anisotropic}
Let $(M,F)$ be a conic pseudo-Finsler surface. The anisotropic conformal change  \eqref{the anisotropic conformal transformation} is  vertical $C$-anisotropic  if and only if $F$ is Riemannian. Moreover, the anisotropic conformal change  \eqref{the anisotropic conformal transformation} is  vertical $\overline{C}$-anisotropic  if and only if $\overline{F}$ is Riemannian.   
\end{proposition}
\begin{proof}
Since, we have, by  \eqref{vanishing li mi},  \eqref{cartan tensor Finsler surface} and \eqref{vertical and horizontal deri of f},
\begin{equation*}
(\dot{\partial}_{i}\phi) C^i_{jk}=\frac{\mathcal{I}}{\overline{F}}(\dot{\partial}_{i}\phi) m^im_jm_k=\frac{\mathcal{I}}{F^2}\varepsilon\phi_{;2}m_{j}m_{k}.
\end{equation*}
 Consequently, the anisotropic conformal transformation \eqref{the anisotropic conformal transformation} is  vertical $C$-anisotropic  if and only if $F$ is Riemannian (as $\phi_{;2}\neq0$).
 Secondly, the anisotropic conformal change  \eqref{the anisotropic conformal transformation} is  vertical $\overline{C}$-anisotropic  if and only if 
 \begin{equation*}
   0=(\dot{\partial}_{i}\phi)\overline{C}^i_{jk}=\phi_{;2}m_i\mathcal{\overline{I}}(m^im_jm_k-\varepsilon\phi_{;2}\ell^im_jm_k)=\varepsilon\phi_{;2}\mathcal{\overline{I}}m_jm_k.  
 \end{equation*}
As $\phi_{;2}\neq0$, the anisotropic conformal transformation \eqref{the anisotropic conformal transformation} is vertical $\overline{C}$-anisotropic if and only if $\overline{F}$ is Riemannian.
\end{proof}
\bigskip
Vector fields in differential geometry, especially concurrent vector fields, are essential for understanding manifold geometry.  Concurrent vector fields play a key role in characterizing geometric properties in Riemannian and Finsler spaces. In 1950, Tachibana established the basis for investigating concurrent vector fields in Finsler spaces  \cite{Tachibana}. Then, Hashiguchi studied the effect of $C$-conformal change on Finsler metrics in 1976, linking them to concurrent vector fields  \cite{Hashiguchi76} . The notion of a semi-concurrent vector field was introduced by N. Youssef and colleagues in 2019 \cite{Youssef semiconcurrent}.   


\begin{definition}\em{\cite{Youssef semiconcurrent}}
A vector field $X^i(x)$ on M is said to be semi-concurrent if it satisfies
\begin{align}
X^iC_{ijk}=0.
\end{align}
\end{definition}

\begin{remark}\em{\cite[Example 3.1.1.1]{Matsumoto 2003} \label{remark of vector field of x only}}
\em{Let $(M,F)$ be a conic pseudo-Finsler surface. The vector field $X^i(x)=A(x,y)\ell^i+B(x,y)m^i$ in $M$ is a function of position only if and only if $A_{;2}=B$ and $B_{;2}=-\varepsilon(A-\mathcal{I}B)$. \\
Similarly, a covariant vector field $w_i(x)=C(x,y)\ell_i+D(x,y) m_i$ is a function of $x$ only if and only if  $C_{;2}=D$ and $D_{;2}=-\varepsilon(C-\mathcal{I}D)$.\\
Also, we have if one of \lq \lq $A$ and $B$" or \lq \lq $C$ and $D$" vanishes then the other vanishes also.}  
\end{remark}

We introduce an alternative proof to the known result that a two-dimensional Finsler manifold admitting a semi-concurrent vector field is Riemannian \cite[Theorem 3.9 (a)]{Youssef semiconcurrent}.
\begin{proposition}\label{F admit semi-concurrent}
 A conic pseudo-Finsler surface  $(M,F)$  admits semi-concurrent vector field if and only if $F$is Riemannian.    
\end{proposition}
\begin{proof}
    Assume that $(M,F)$ admits semi-concurrent vector field $X^i(x)=A(x,y)\ell^i+B(x,y)m^i$, then we have
    \begin{equation*}
        0=X^i(x)C_{ijk}=\frac{\mathcal{I}}{F}(A(x,y)\ell^i+B(x,y)m^i)m_im_jm_k=\frac{\varepsilon \mathcal{I}}{F}B(x,y)m_jm_k .
    \end{equation*}
   Which is equivalent to either $F$ is Riemannian or $B(x,y)=0$. From Remark \ref{remark of vector field of x only}, the later shows that $X^i$ is zero vector, which is contradiction. Hence the result.
\end{proof}
   
Now, we will determine the necessary and sufficient condition for the property of admitting a semi-concurrent vector field to be preserved under the anisotropic conformal change  \eqref{the anisotropic conformal transformation}.  
\begin{proposition}\label{theorem of F bar admits concurrent vector}
    Let \eqref{the anisotropic conformal transformation} be the anisotropic conformal transformation of a Finsler space $(M,F)$ which  admits a semi-concurrent vector field. A necessary and sufficient condition for  $(M,\overline{F})$ to admit a semi-concurrent vector field is that the function $(4\phi-\ln\rho)$ is  isotropic.
\end{proposition}
\begin{proof}
Let $(M,F)$ admits a semi-concurrent vector field,  by Proposition \ref{F admit semi-concurrent},  we get $\mathcal{I} =0$. Also,    $(M,\overline{F})$ admits a semi-concurrent vector field if and only if  $\overline{F}$ is Riemannian, or equivalently $\overline{\mathcal{I}}=0$, that is
\begin{align}\label{concurrent vector relation 1}
\mathcal{I}+2\varepsilon\phi_{;2}-\frac{\varepsilon\rho_{;2}}{2\rho}=0 .
\end{align}
Therefore, \eqref{concurrent vector relation 1} becomes  $2\varepsilon\phi_{;2}-\frac{\varepsilon\rho_{;2}}{2\rho}=0$,  which is $(4\phi-\ln\rho)_{;2}=0$. Then,  $(M,\overline{F})$ admits a semi-concurrent vector field is equivalent to   $(4\phi-\ln\rho)$ is an isotropic function.
\end{proof}

\begin{remark}\label{remark of necc and suff of F bar admits semicouncurrent}
\begin{description}
    \item[(i)] From \eqref{formula of rho} and Proposition \ref{theorem of F bar admits concurrent vector} a necessary and sufficient condition for $(M,\overline{F})$ to  admit a semi-concurrent vector field is that 
$\phi_{;2;2;2}+6\phi_{;2}\phi_{;2;2}+4(\phi_{;2})^3+4\varepsilon\phi_{;2}=0,$ provided that, $(M,F)$ admits a semi-concurrent vector field.
    \item[(ii)] If $\phi$ is a function of position only, the property of admitting a semi-concurrent vector field is conserved. In other words, the property of admitting a semi-concurrent vector field is preserved if the anisotropic conformal change is reduced to an isotropic conformal change.
\end{description} 
\end{remark}

\begin{corollary}
Let $(M,F)$ be a  conic pseudo-Finsler surface and \eqref{the anisotropic conformal transformation} be a proper anisotropic conformal transformation. If  $ \overline{\mathfrak{g}}=e^{4\phi}\mathfrak{g}$, then the Finsler metric $\overline{F}$ admits a semi-concurrent vector if and only if $\phi_{;2;2}=(\phi_{;2})^2.$ 
\end{corollary}
\begin{proof}
Since $ \overline{\mathfrak{g}}=e^{4\phi}\mathfrak{g}$ is equivalent to $\sigma=(\phi_{;2})^2$, then from \eqref{formula of sigma} we get
\begin{align}\label{concurrent vector relation 3}
\mathcal{I}+\varepsilon\phi_{;2}=\frac{-\varepsilon\phi_{;2;2}}{\phi_{;2}}.
\end{align}
Also, by using \eqref{formula of rho} we have $\rho=\varepsilon,\quad\rho_{;2}=0.$
Moreover,  Remark \ref{remark of necc and suff of F bar admits semicouncurrent} and\eqref{concurrent vector relation 3} gives rise to the Finsler metric $\overline{F}$ admits a semi-concurrent vector if and only if 
\begin{align*}
0=\mathcal{I}+2\varepsilon\phi_{;2}-\frac{\varepsilon\rho_{;2}}{2\rho}=\frac{-\varepsilon\phi_{;2;2}}{\phi_{;2}}+\varepsilon\phi_{;2}=\frac{\varepsilon(-\phi_{;2;2}+(\phi_{;2})^2)}{\phi_{;2}}.
\end{align*}
Hence, $\overline{F}$ admits a semi-concurrent vector if and only if $\phi_{;2;2}=(\phi_{;2})^2.$ 
\end{proof}

\section{Anisotropic $\phi T$-condition}\label{section_phi-$T$-condition}
The $T$-tensor \lq \lq $FT_{ijkh}=\mathcal{I}_{;2} \,m_im_jm_km_h$"  is one of the fundamental objects in Finsler geometry.  The Finsler metric satisfies the $T$-condition if its $T$-tensor vanishes identically. The $T$-condition and $\boldsymbol{\sigma} T$-condition are equivalent in the case of Finsler surfaces \cite[Theorem 3.5]{elgendi}. 
\begin{definition}
   A Finsler space $ (M,F)$ satisfies the $\boldsymbol{\sigma} T$-condition if M admits a non-constant smooth function  $\boldsymbol{\sigma}$ such that $(\partial_{i}\boldsymbol{\sigma})T^i_{jkr}=0$.
\end{definition}

We present an alternative proof of \cite[Theorem 3.5]{elgendi}

\begin{proposition}\label{sigma-t-condition}
 A two-dimensional Finsler metric satisfies  $\boldsymbol{\sigma} T$-condition if and only if has vanishing $T$-tensor.
\end{proposition}
\begin{proof}
   Let $(M,F)$ satisfy  $\boldsymbol{\sigma} T$-condition, that is,  $( \partial_{i} \boldsymbol{\sigma})T^i_{jkr}=0$. Then,  by  \eqref{vertical and horizontal deri of f}, \eqref{vanishing li mi} and the definition of $T$-tensor,  we have 
   \begin{equation}
    0 =\partial_i\boldsymbol{\sigma} =(\delta_{i}\boldsymbol{\sigma})T^i_{jkr}=(\boldsymbol{\sigma}_{,1}\ell_{i}+\boldsymbol{\sigma}_{,2}m_i)\frac{\mathcal{I}_{;2}}{F}m^im_jm_km_r=\frac{\varepsilon\mathcal{I}_{;2}\boldsymbol{\sigma}_{,2}}{F}m_jm_km_r .
  \end{equation}
 Consequently,  the  Finsler metric $F$ satisfies $\boldsymbol{\sigma} T$-condition if and only if either  $F$ has vanishing $T$-tensor (i.e., $\mathcal{I}_{;2}=0$) or $\boldsymbol{\sigma}_{,2}=0$. The later gives $\boldsymbol{\sigma}_{,2}=\boldsymbol{\sigma}_{;2}=0$. Then, $\boldsymbol{\sigma}$ is a constant function,  by Lemma \ref{lemma of condition to homthety},  which contradicts $\boldsymbol{\sigma}$ is non-constant function. 
\end{proof}
From Proposition \ref{sigma-t-condition}, any Landsberg surface  satisfies the $\boldsymbol{\sigma} T$-condition is Berwaldian \cite{elgendi}.  Under the anisotropic conformal transformation \eqref{the anisotropic conformal transformation}, we see that the conformal factor depends on both position and direction arguments and the transformation of the $T$-tensor is not invariant:
\begin{align*}
\overline{T}_{ijhk}=&\frac{\varepsilon e^{3\phi}}{ \rho}\left[ T_{ijkh}+\frac{1}{2F\rho}\left(4\varepsilon\rho\phi_{;2;2}+\rho_{;2}(\mathcal{I}+2\varepsilon\phi_{;2}+\frac{\varepsilon\rho_{;2}}{2\rho})-\varepsilon\rho_{;2;2}\right) m_im_jm_hm_k\right].
\end{align*}
This leads us to define $\phi T$-condition, $\phi \overline{T}$-condition, horizontal $\phi T$-condition and horizontal $\phi \overline{T}$-condition and determine the conditions under which these properties coincide with the $T$-condition. In other words, we aim to find the conditions such that any Landsberg surface that satisfies these properties is a Berwald surface. In Finsler surfaces, the T-tensor components $ T^i_{jkr}$ are defined by $FT^i_{jkr}=\mathcal{I}_{;2}m^im_jm_km_r$. By using \eqref{transform of Berwald frame}, under the anisotropic conformal transformation \eqref{the anisotropic conformal transformation}. We have
\begin{align}\label{transform_(1,3)T_tensor}
\overline{F}\;\overline{T}^i_{jkr}=\mathcal{\overline{I}}_{;b}\overline{m}^i\overline{m}_j\overline{m}_k\overline{m}_r=e^{2\phi}\sqrt{\frac{\varepsilon}{\rho}}\mathcal{\overline{I}}_{;2}(m^i-\varepsilon\phi_{;2}\ell^{i})m_jm_km_r
\end{align} 
Also, from \cite[~\S 2]{second_paper} we have
\begin{align}
\overline{\mathcal{I}}_{;\,b}&=\sqrt{\varepsilon\rho} \;\overline{\mathcal{I}}_{;2},\label{I;b bar formula}\\
\overline{\mathcal{I}}_{,\,a}&=e^{-\phi} \,[\overline{\mathcal{I}}_{,1}-\frac{2\varepsilon}{F^2} Q\,\overline{\mathcal{I}}_{;2}],\label{I,a bar formula}\\
\overline{\mathcal{I}}_{,\,b}&= e^{-\phi}\sqrt{\varepsilon\rho}\;[\overline{\mathcal{I}}_{,2}-\phi_{;2}\,\overline{\mathcal{I}}_{,1}-\frac{\varepsilon}{F^2}(\varepsilon P+Q_{;2}-\varepsilon\mathcal{I}\,Q-2\phi_{;2}\,Q)\overline{\mathcal{I}}_{;2}]\label{I,b bar formula},
\end{align}
where 
\begin{align}
\overline{\mathcal{I}}_{;2}&=\frac{\sqrt{\varepsilon\rho}}{2\rho}\,[\rho_{;2}(\mathcal{I}+2\varepsilon\,\phi_{;2}+\frac{\varepsilon}{2}(\ln \rho)_{;2})+2\rho(\mathcal{I}_{;2}+2\varepsilon\phi_{;2;2})-\varepsilon\rho_{;2;2}].\label{I; 2 bar formula}\\
\overline{\mathcal{I}}_{,1}&=\frac{\sqrt{\varepsilon\rho}}{2\rho}\,[\rho_{,1}(\mathcal{I}+2\varepsilon\phi_{;2}+\frac{\varepsilon}{2} (\ln \rho)_{;2})+2\rho(\mathcal{I}_{,1}+2\varepsilon\phi_{;2,1})-\varepsilon\rho_{;2,1}].\label{I,1 bar formula}\\
\overline{\mathcal{I}}_{,2}&=\frac{\sqrt{\varepsilon\rho}}{2\rho}\,[\rho_{,2}(\mathcal{I}+2\varepsilon\phi_{;2}+\frac{\varepsilon}{2}(\ln \rho)_{;2})+2\rho(\mathcal{I}_{,2}+2\varepsilon\phi_{;2,2})-\varepsilon\rho_{;2,2}].\label{I,2 bar formula}
\end{align}
The Landsbergian property of $\overline{F}$ is characterized by $\overline{\mathcal{I}}_{,a}=0$ and the Berwaldian property of $\overline{F}$ is characterized by $\overline{\mathcal{I}}_{,a}=0=\overline{\mathcal{I}}_{,b}.$
\begin{definition}
The anisotropic conformal change \eqref{the anisotropic conformal transformation} is said to be satisfying the anisotropic $\phi T$-condition if $(\partial_{i}\phi)T^i_{jkr}=0$. Similarly, the anisotropic conformal change \eqref{the anisotropic conformal transformation} is said to be satisfying the anisotropic $\phi \overline{T}$-condition if $(\partial_{i}\phi)\overline{T}^i_{jkr}=0$.
\end{definition}

\begin{lemma}\label{lemma anisotropic phi T-condition}
  The anisotropic conformal change \eqref{the anisotropic conformal transformation} satisfies the anisotropic $\phi T$-condition if and only if either $m^i\partial_i\phi=0$ or $F$ has vanishing $T$-tensor.  Similarly, The anisotropic conformal change \eqref{the anisotropic conformal transformation} satisfies the anisotropic $\phi \overline{T}$-condition  if and only if  either $m^i\partial_i\phi=\varepsilon\phi_{;2}\ell^{i}\partial_i\phi$ or $\overline{F}$ has vanishing $\overline{T}$-tensor.
\end{lemma}   
\begin{proof}
 The anisotropic conformal change \eqref{the anisotropic conformal transformation} satisfies the anisotropic $\phi T$-condition, by definition,  means that
\begin{equation*}
    0=(\partial_{i}\phi) T^i_{jkr}=\frac{\mathcal{I}_{;2}}{F}m_jm_km_r(m^i\partial_i\phi).
\end{equation*}
Which is equivalent to either $F$ has vanishing $T$-tensor or $m^i\partial_i\phi=0$. Similarly, from \eqref{transform_(1,3)T_tensor}, \eqref{the anisotropic conformal transformation} satisfies the anisotropic $\phi \overline{T}$-condition is equivalent to
  \begin{equation*}
   0= (\partial_{i}\phi)\overline{T}^i_{jkr}=e^{\phi}\sqrt{\frac{\varepsilon}{\rho}}\frac{\mathcal{\overline{I}}_{;2}}{F}(m^i\partial_i\phi-\varepsilon\phi_{;2}\ell^{i}\partial_i\phi)m_jm_km_r .
\end{equation*} 
Thereby,  the anisotropic $\phi \overline{T}$-condition is satisfied if and only if either $m^i\partial_i\phi=\varepsilon\phi_{;2}\ell^{i}\partial_i\phi$ or $\mathcal{\overline{I}}_{;2}=0$.  Then,  either  $\overline{F}$ has vanishing  $\overline{T}$-tensor, by \cite[Remark 4.1]{second_paper},  or $m^i\partial_i\phi=\varepsilon\phi_{;2}\ell^{i}\partial_i\phi$ .
\end{proof}
\begin{theorem}\label{in unicorn 1, 2}
    Let $\overline{F}= e^{\phi} F$ be the anisotropic conformal transformation  \eqref{the anisotropic conformal transformation}. Then we have:
    \begin{description}
        \item[(i)] If $F$ is a Landsberg metric  and \eqref{the anisotropic conformal transformation} satisfies the anisotropic $\phi T$-condition, provided that, $m^i\partial_i\phi\neq0$, then $F$ is  Berwaldian.  
    \item[(ii)] If $(M,\overline{F})$ is a Landsberg surface and \eqref{the anisotropic conformal transformation} satisfies the anisotropic $\phi \overline{T}$-condition, provided that, $m^i\partial_i\phi\neq\varepsilon\phi_{;2}\ell^{i}\partial_i\phi$, then $(M,\overline{F})$ is Berwaldian.  In other words, a non-Riemannian Landsberg surface  $(M,\overline{F})$  is  Berwaldian if  the anisotropic $\phi \overline{T}$-condition is satisfied,  provided that,  \eqref{the anisotropic conformal transformation} does not  $\overline{C}$-anisotropic conformal change.
    \end{description} 
\end{theorem}
\begin{proof} 
\textbf{(i)} and first part of \textbf{(ii)} Follows from Lemma \ref{lemma anisotropic phi T-condition} and the fact that any Landsberg surface with a vanishing $T$-tensor is Berwaldian. The proof of second part of  \textbf{(ii)}:
    As, \eqref{the anisotropic conformal transformation}  does not satisfy  $\overline{C}$-anisotropic conformal change and $(M,\overline{F})$ is a non-Riemannian surface, then from Proposition \ref{proposition1 of bar$FC$-conformal}, we get
    \begin{equation}\label{phi T bar 1}
     m^i\partial_i\phi\neq\varepsilon\phi_{;2}\ell^{i}\partial_i\phi   .
    \end{equation}
    Since \eqref{the anisotropic conformal transformation} satisfies the anisotropic $\phi \overline{T}$-condition,  then from Lemma  \ref{lemma anisotropic phi T-condition} and \eqref{phi T bar 1}, we get $\overline{\mathcal{I}}_{;2}=0$, consequently, $(M,\overline{F})$ has vanishing $\overline{T}$-tensor. Hence $(M,\overline{F})$ is Berwaldian.    
\end{proof}

\begin{remark}
    Every $C$-anisotropic conformal transformation satisfies the $\phi T$-condition.  Correspondingly,  every $\overline{C}$-anisotropic conformal transformation satisfies the $\phi \overline{T}$-condition.  However, the converse is not true in general, that is,  not every anisotropic conformal transformation satisfies the $\phi T$-condition is $C$-anisotropic conformal transformation.
\end{remark}

\begin{definition}
The anisotropic conformal transformation \eqref{the anisotropic conformal transformation} is said to be satisfying the horizontal $\phi T$-condition if $(\delta_{i}\phi)T^i_{jkr}=0.$ Similarly, the anisotropic conformal transformation \eqref{the anisotropic conformal transformation} is said to be satisfying the horizontal $\phi\overline{T}$-condition if $(\delta_{i}\phi)\overline{T}^i_{jkr}=0.$  
\end{definition}

\begin{proposition}\label{proposition-horizotal-phi_T-condition}
Let $(M,F)$ be a conic pseudo-Finsler surface. Then we have the following:
\begin{description} 
\item[(i)] The anisotropic conformal transformation \eqref{the anisotropic conformal transformation} satisfies horizontal $\phi T$-condition if and only if either $F$ has vanishing $T$-tensor or $\phi_{,2}=0.$
\item[(ii)]  The anisotropic conformal transformation \eqref{the anisotropic conformal transformation} satisfies  horizontal $\phi\overline{T}$-condition if and only if $\overline{F}$ has vanishing $\overline{T}$-tensor or $\phi_{,2}=\phi_{,1}\phi_{;2}.$ 
\end{description}
\end{proposition}
\begin{proof}
\begin{description}
    \item[(i)] Assume that the anisotropic conformal transformation \eqref{the anisotropic conformal transformation} satisfies horizontal $\phi T$-condition, then we have 
\begin{equation*}
 0=(\delta_{i}\phi)T^i_{jkr}=(\phi_{,1}\ell_i+\phi_{,2}m_i)\mathcal{I}_{;2}m^im_jm_km_r=\varepsilon\phi_{,2}\mathcal{I}_{;2}m_jm_km_r,
\end{equation*}
which is equivalent to that either $F$ has vanishing $T$-tensor or $\phi_{,2}=0.$
\item[(ii)] From \eqref{transform_(1,3)T_tensor}, the anisotropic conformal transformation \eqref{the anisotropic conformal transformation}  satisfies  horizontal $\phi\overline{T}$-condition if and only if
\begin{align}
0=&(\delta_{i}\phi)\overline{T}^i_{jkr}=[(\phi_{,1}\ell_i+\phi_{,2}m_i)\mathcal{\overline{I}}_{;2}(m^im_jm_km_r-\varepsilon\phi_{;2}\ell^im_jm_km_r)]\nonumber\\
=&\mathcal{\overline{I}}_{;2}(\varepsilon \phi_{,2}m_jm_km_r-\varepsilon\phi_{,1}\phi_{;2}m_jm_km_r)].\nonumber
\end{align}
Hence, the anisotropic conformal transformation \eqref{the anisotropic conformal transformation}  satisfies  horizontal $\phi\overline{T}$-condition if and only if $\overline{F}$ has vanishing $\overline{T}$-tensor  or $\phi_{,2}=\phi_{,1}\phi_{;2}.$
\end{description}
\vspace*{-0.9 cm}\[\qedhere\]
\end{proof}

\begin{theorem}
    Assume that \eqref{the anisotropic conformal transformation} is a vertical $C$-anisotropic conformal change with $\phi_{;2}$ being horizontally constant. Then the following are equivalent:
    \begin{description}
        \item[(i)] The Finsler metric $\overline {F}$ is Landsbergian. 
        \item[(ii)] The anisotropic conformal change \eqref{the anisotropic conformal transformation} satisfies the horizontal $\phi\overline{T}$-condition.
        \item[(iii)] The Finsler metric $\overline {F}$ is Berwaldian or the two geodesic sprays $S$ and $\overline{S}$ are projectively equivalent.
    \end{description}
\end{theorem}
\begin{proof}
We assume that \eqref{the anisotropic conformal transformation} is a vertical $C$-anisotropic change. By Proposition \ref{prposition vertical $C$-anisotropic}, $F$ is a Riemannian metric and $\phi_{;2}$ is horizontally constant, by assumption. Then, according to \cite[Corollary 2.5]{second_paper}, we have that $\rho$ and $\rho_{;2}$ are horizontally constant. This implies that
\begin{equation}\label{horizontally_constant_rho}
\mathcal{I}=\rho_{,1}=\rho_{,2}=\rho_{;2,1}=\rho_{;2,2}=0.
\end{equation}
 \textbf{(i)} $\Longleftrightarrow$ \textbf{(ii)} It follows by \cite[Proposition 4.7]{second_paper} and Proposition \ref{proposition-horizotal-phi_T-condition} \textbf{(ii)}.\\
\textbf{(ii)}$\Longrightarrow$ \textbf{(iii)} 
Assume that \eqref{the anisotropic conformal transformation} satisfies the horizontal $\phi\overline{T}$-condition. By Proposition \ref{proposition-horizotal-phi_T-condition}, we have $\mathcal{\overline{I}}_{;2}=0$ or $\phi_{,2}=\phi_{;2}\phi_{,1}.$
First, if $\mathcal{\overline{I}}_{;2}=0$, substituting from \eqref{horizontally_constant_rho} into \eqref{I,1 bar formula} and \eqref{I,2 bar formula}, we get 
\begin{equation}\label{constant_main_scalar_bar}
    \mathcal{\overline{I}}_{;2}=0=\mathcal{\overline{I}}_{,1}=\mathcal{\overline{I}}_{,2}.
\end{equation}
Substituting \eqref{constant_main_scalar_bar} into \eqref{I,a bar formula} and \eqref{I,b bar formula}, we find that $\overline{F}$ is  Berwaldian. Secondly, if $\phi_{;2}\phi_{,1}=\phi_{,2}$ and $\phi_{;2}$ is horizontally constant, by assumption and using the commutation formula \eqref{second comutation}, we have 
\begin{equation}
    0=\phi_{;2}\phi_{,1}+\phi_{;2,1}-\phi_{,2}=\phi_{;2}\phi_{,1}+\phi_{,1;2}-\phi_{,2}=Q.
\end{equation}
\textbf{(iii)}$\Longrightarrow$ \textbf{(i)} Let $\overline{F}$ be a Berwald metric then $\overline{F}$ Landsbergian. Also, if the two geodesic spray $S$ and $\overline{S}$ are projectively equivalent,    using \eqref{horizontally_constant_rho}, we get 
\begin{equation}\label{Q_Ibar_1_vanish}
    Q=\mathcal{\overline{I}}_{,1}=0.
\end{equation}
Substituting \eqref{Q_Ibar_1_vanish} into \eqref{I,a bar formula}, we find that $\overline{F}$ is Landsbergian.
\end{proof}

\begin{definition}
The proper anisotropic conformal change \eqref{the anisotropic conformal transformation} is said to be satisfying the vertical $\phi T$-condition if $(\dot{\partial}_i\phi)T^i_{jkr}=0$.  Similarly,  \eqref{the anisotropic conformal transformation} is said to be satisfying the vertical $\phi\overline{T}$-condition if $(\dot{\partial}_i\phi)\overline{T}^i_{jkr}=0.$
\end{definition}
\begin{lemma}\label{lemma_vertical_phi_T-condition}
Let $(M,F)$ be a conic pseudo-Finsler surface. The anisotropic conformal change  \eqref{the anisotropic conformal transformation} satisfies the vertical $\phi T$-condition if and only if $F$ has a vanishing $T$-tensor.  
\end{lemma}
\begin{proof}
We have, by  \eqref{vanishing li mi}, 
\begin{equation*}
(\dot{\partial}_{i}\phi) T^i_{jkr}=\frac{\mathcal{I}_{;2}}{F}(\dot{\partial}_{i}\phi) m^im_jm_km_r=\frac{\varepsilon\mathcal{I}_{;2}}{F}\phi_{;2}m_{j}m_{k}m_r.
\end{equation*}
 Consequently, the anisotropic conformal transformation \eqref{the anisotropic conformal transformation} satisfies the vertical $\phi T$-condition if and only if $F$ has a vanishing $T$-tensor.
\end{proof}
According to Lemma \ref{lemma_vertical_phi_T-condition}, we have
\begin{remark}\label{remark_vertical_phi_condition}
\begin{description}
    \item[(i)]The anisotropic conformal transformation \eqref{the anisotropic conformal transformation} satisfies the vertical $\phi\overline{T}$-condition if and only if $\overline{F}$ has a vanishing $\overline{T}$-tensor.
    \item[(ii)] The $\phi T$-condition is equivalent to the $T$-condition. 
\end{description}
\end{remark}
From Remark \ref{remark_vertical_phi_condition} and the fact that any Landsberg surface satisfies the $T$-condition is Berwaldian, we have 
\begin{corollary}
    Let $F$ be anisotropically conformal to $\overline{F}$ by \eqref{the anisotropic conformal transformation}. Then, we have the following:
    \begin{description}
        \item[(i)] If $F$ is a Landsberg metric and  \eqref{the anisotropic conformal transformation} satisfies the vertical $\phi T$-condition, then, $F$ is Berwaldian. 
        \item[(ii)] If $\overline{F}$ is a Landsberg metric and \eqref{the anisotropic conformal transformation} satisfies the vertical $\phi\overline{T}$-condition, then, $\overline{F}$ is Berwaldian. 
    \end{description}
\end{corollary}

\section{Special choice of the conformal factor}\label{special conformal factor}

In this section, we assume that $F$ is a non-Riemannian metric and the conformal factor of the anisotropic conformal transformation \eqref{the anisotropic conformal transformation} is the main scalar $\mathcal{I}$ of $F$, that is,  
\begin{equation}\label{conformal factor I}
\overline{F}(x,y)=e^{\mathcal{I}(x,y)}F(x,y).
\end{equation}
From \eqref{formula of sigma}-\eqref{formula of P only} then we have the following:
 \begin{align}
  \sigma&=\mathcal{I}_{;2;2}+\varepsilon \mathcal{I}\mathcal{I}_{;2}+2\,(\mathcal{I}_{;2})^{2},\label{formula of sigma phi=I}\\
2Q&=\varepsilon\rho F^2(\mathcal{I}_{;2}\mathcal{I}_{,1}+\mathcal{I}_{,1;2}-2\mathcal{I}_{,2}),\label{formula of Q only phi=I}\\
2P&=-\rho F^2\mathcal{I}_{;2}(\mathcal{I}_{;2}\mathcal{I}_{,1}+\mathcal{I}_{,1;2}-2\mathcal{I}_{,2})+F^2\mathcal{I}_{,1}.\label{formula of P only phi=I}
\end{align}
Furthermore, from \eqref{formula of Q only phi=I} and \eqref{formula of P only phi=I}, we get
\begin{equation}\label{relation between P and Q phi=I}
2\varepsilon \mathcal{I}_{;2}Q+2P=F^2\mathcal{I}_{,1}.
\end{equation}

 \begin{proposition}
 Under  the anisotropic conformal change \eqref{conformal factor I},  we have the following:
      \begin{description}
        \item[(i)] If the conic pseudo-Finsler surface $(M,F)$ has vanishing $T$-tensor,  then the anisotropic conformal change \eqref{conformal factor I} reduces to isotropic conformal change. 
        \item[(ii)] If $(M,F)$ is a Berwald surface, then the geodesic spray is invariant.
        \item[(iii)] If $(M,F)$ is a Landsberg surface, then the geodesic spray of $\overline{F}$  has the the coefficients \\ $\overline{G}^i=G^i-\varepsilon\rho F^2\mathcal{I}_{,2}m^i+\rho F^2\mathcal{I}_{;2}\mathcal{I}_{,2}\ell^i.$  
    \end{description}  
    \end{proposition}
\begin{proof}
    \begin{description}
\item[(i)] It follows from the fact that a Finsler $F$ has a vanishing $T$-tensor if and only if the main scalar (conformal factor) is a function of $x$ only. 
\item[(ii)] Let $(M,F)$ be a Berwald surface, then $\mathcal{I}_{,1}=\mathcal{I}_{,2}=0$. Substituting into \eqref{formula of Q only phi=I} and \eqref{formula of P only phi=I}, we get $P=Q=0$. Then, by \eqref{transform of coefficient spray}, the geodesic spray is invariant under the anisotropic change.
\item[(iii)] Since $F$ is  Landsbergian, that is,  $\mathcal{I}_{,1}=0$.  Thus,   \eqref{formula of Q only phi=I} and \eqref{formula of P only phi=I} become $Q=- \varepsilon\rho F^2 \mathcal{I}_{,2}$ and 
$P=\rho F^2\mathcal{I}_{;2}\,\mathcal{I}_{,2}$.  Hence, the result follows from \eqref{transform of coefficient spray}.
    \end{description}
    \vspace*{-0.9 cm}\[\qedhere\]
\end{proof}
 
\begin{lemma}
    Assume that the anisotropic conformal change \eqref{conformal factor I} satisfies the $C$-anisotropic conformal change. Then, $F$ is weakly Berwaldian if and only if $G^i_km^km_i=0.$
\end{lemma}
\begin{proof}
    From Proposition \ref{proposition1 of bar$FC$-conformal} and \eqref{equivalence 2 of C bar}, the anisotropic conformal transformation \eqref{conformal factor I} satisfies the $C$-anisotropic conformal change if and only if $\mathcal{I}_{,2}=-\frac{\varepsilon\phi_{;2}}{F}G^i_{k}m^k m_i.$ Then $F$ is weakly Berwald if and only if $G^i_km^km_i=0.$
\end{proof}
Let $\overline{F}$ be a non-Riemannian metric. From Proposition \ref{proposition1 of bar$FC$-conformal}: the property of $C$-anisotropic conformal is invariant if and only if
$$\ell^i\partial_i\mathcal{I}=0\Longleftrightarrow \mathcal{I}_{,1}=-\frac{2}{F^2}\mathcal{I}_{;2}G^k\;m_k .$$ 

\begin{lemma}
 Under  the anisotropic conformal change \eqref{conformal factor I},     if  $\overline{F}$ is  a non-Riemannian metric and the property of $C$-anisotropic conformal is invariant,  then, the Finsler metric $F$ is Landsbergian if and only if it is projectively flat metric.
\end{lemma}
\begin{theorem}
    The Finsler surface $(M,F)$ is Berwaldian if one of the following is  satisfied:
    \begin{description}
        \item[(i)]$(M,F)$ is a Landsberg surface and \eqref{conformal factor I} is horizontal $C$-anisotropic conformal transformation.
        \item[(ii)]$(M,F)$ is a Landsberg or weak Berwald  surface and \eqref{conformal factor I} is horizontal $\overline{C}$-anisotropic conformal transformation.
    \end{description}
\end{theorem}
\begin{proof}
\begin{description}
    \item[(i)] Since $(M,F)$ is a Landsberg surface (i.e., $\mathcal{I}_{,1}=0$) and \eqref{conformal factor I} is a horizontal $C$-anisotropic conformal transformation, then  from Proposition \ref{theorem-horizotal-$C$-anisotropic}, we have $\mathcal{I}_{,2}=0$. Thus, $(M,F)$ is a Berwald surface. 
    \item[(ii)] From Proposition \ref{theorem-horizotal-$C$-anisotropic}, the anisotropic conformal change \eqref{conformal factor I} is horizontal $\overline{C}$-anisotropic  if and only if $\mathcal{I}_{,2}=\mathcal{I}_{;2}\mathcal{I}_{,1}$. Then  $(M,F)$ is a Berwald surface if  $(M,F)$ is a Landsberg or weak Berwald surface.   
\end{description}   
\vspace*{-0.9 cm}\[\qedhere\]
\end{proof}

 \section{A Finslerian Schwarzschild-de Sitter solution } \label{Sec_6}
    An anisotropic conformal transformation differs from a conformal transformation in which it can transform a Riemannian metric into a Finslerian metric \cite{first_paper, second_paper}. For instance, consider the Riemannian metric defined on the 2-dimensional sphere by:
$$F(\theta ,\eta ; y^{\theta}, y^{\eta})= \sqrt{(y^{\theta})^2  + \sin^2 (\theta) \, (y^{\eta})^2}.$$ By using the conformal factor 
$$\phi = \ln \left(
\frac{ \sqrt{(1 - a^2 \sin^2 \theta)(y^\theta)^2 + \sin^2 \theta \, (y^\eta)^2} - a \sin^2 \theta \, y^\eta }{ (1 - a^2 \sin^2 \theta)\sqrt{(y^\theta)^2 + \sin^2 \theta \, (y^\eta)^2} }
\right),$$
we get the Finsler metric 
\begin{equation}\label{Finslerian_sphere}
   \overline{F}=e^{\phi}F = \frac{\sqrt{(1-a^2 \,\sin^2 (\theta) )\, (y^{\theta})^2  +\sin^2 (\theta) \, (y^{\eta})^2}}{1-a^2 \sin^2 (\theta) } -\frac{a \sin^2 (\theta) \, y^\eta}{1-a^2 \sin^2 (\theta)},
\end{equation}
where $0\leq a<1$ is a constant known as the Finsler parameter. It is worth noting that $\overline{F}$ is a two dimensional Randers-Finsler space with constant positive
flag curvature $\lambda = 1$ and is homotopy equivalent to the two dimensional sphere with volume $4\pi$, when $a=0$, $\overline{F}=F$. 
The Finslerian analogue of Birkhoff’s theorem states that a Finslerian gravitational field with the symmetry of a \textquotedblleft Finslerian sphere\textquotedblright in vacuum must be static \cite{PRD2014}. Building on this, the Einstein field equations with a non-zero cosmological constant ($\Lambda \neq 0$) have been extended to the Finslerian framework in \cite{Finsler gravity2024}, resulting in a Finslerian Schwarzschild-de Sitter solution that possesses the symmetry of a \textquotedblleft Finslerian sphere\textquotedblright. An example of such a Finslerian sphere is given by the expression \eqref{Finslerian_sphere}.
 We have the following:
 \begin{description}
     \item[(i)] The Finsler metric $F$ is a Riemannian metric of constant curvature and does not satisfy the Hamel’s equations, therefore it is not locally projectively flat.
     \item[(ii)] $F$ is $C$-anisotropic conformal, horizontal $C$-anisotropic, vertical $C$-anisotropic and satisfies the $\phi T$-condition, as $F$ is a Riemannian metric.
     \item[(iii)] In a non-Riemannian Randers space $(M,\overline{F})$, the 1-form is not covariantly constant with respect to the Riemannian metric $\alpha =\frac{\sqrt{(1-a^2 \,\sin^2 (\theta) )\, (y^{\theta})^2  +\sin^2 (\theta) \, (y^{\eta})^2}}{1-a^2 \sin^2 (\theta) } .$ Hence, it is neither Berwaldian nor Landsbergian (by using \cite[Theorem 3.2]{Shenbook16}.)  \\In more details: the Riemannian metric  $a_{ij}$  and its inverse $a^{ij}$ are given by
$$(a_{ij}) = 
\begin{pmatrix}
\frac{1}{ 1 - a^2 \sin^2 \theta } & 0 \\
0 & \frac{\sin^2 \theta}{(1 - a^2 \sin^2 \theta)^2}
\end{pmatrix},
\qquad (a^{ij}) = 
\begin{pmatrix}
1 - a^2 \sin^2 \theta & 0 \\
0 & \frac{(1 - a^2 \sin^2 \theta)^2}{\sin^2 \theta}
\end{pmatrix}.$$
  The corresponding non-zero coefficients $\gamma^h_{ij}$  of the Levi-Civita connection   are given by:
  $$\gamma^\theta_{\theta\theta}=  \frac{  a^2 \sin\theta\,  \cos \theta }{ (1 - a^2 \sin^2 \theta)}    ,\quad \gamma^\eta_{\theta\eta}=  \frac{  \cos \theta (1 + a^2 \sin^2 \theta)}{\sin{\theta} \, (1 - a^2 \sin^2 \theta)}    ,\quad  \gamma^\theta_{\eta\eta}= -\frac{ \cos \theta\, \sin{\theta \, (1 + a^2 \sin^2 \theta)}}{  (1 - a^2 \sin^2 \theta)^2}    .$$
Since,   the 1-form \( \beta = b_i \, y^i \),  with $b_{\theta}=0,\, b_{\eta}=\frac{-a\sin{\theta}}{1-a^2\sin^2{\theta}}$,  the covariant derivative  of $b_\theta$ with respect to the Levi-Civita connection is given by:
\begin{align*}
\nabla_\eta b_\theta =& \partial_\eta b_\theta -\gamma^\theta_{\theta\eta} b_\theta- \gamma^\eta_{\theta\eta} b_\eta
=-  \frac{-a\sin{\theta}}{(1-a^2\sin^2{\theta})}\frac{  \cos \theta (1 + a^2 \sin^2 \theta)}{\sin{\theta} \, (1 - a^2 \sin^2 \theta)} = \frac{a \cos \theta (1 + a^2 \sin^2 \theta) }{ (1 - a^2 \sin^2 \theta )^2 } .
\end{align*}
Therefore,  the $1$-form is not covariantly constant with respect to the Riemannian metric $\alpha$.

     \item[(iv)]  The Finsler metric $\overline{F}$ is of constant curvature $\lambda=1$ and not locally projectively flat \cite[Theorem 1.1]{Shen_2_dim_randers}. 
    \item[(v)]By using Mapple’s package uploaded to
    
\url{https://github.com/salahelgendi/Special_anisotropics_Section_6}    

     we get $m^i\partial_i\phi-\varepsilon \phi_{;2}\ell^i\partial_i\phi\neq0$ and $\phi_{,2}\neq\phi_{;2}\phi_{,1}$. According to Proposition \ref{proposition1 of bar$FC$-conformal} and Theorem \ref{theorem-horizotal-$C$-anisotropic}, the anisotropic conformal change \eqref{Finslerian_sphere} is not $\overline{C}$-anisotropic conformal and does not satisfy the horizontal $\overline{C}$-anisotropic property,  as $\overline{F}$ is not Riemannian. Also, \eqref{Finslerian_sphere} is not a vertical $\overline{C}$-anisotropic conformal change.
     \item[(vii)] As $F$ is a Riemannian metric, it admits a semi-concurrent vector field.  However, $\overline{F}$ is not (as it is a non-Riemannian metric).
 \end{description}

\section{Concluding Remarks}\label{Concluding Remarks}
In conformal theory, special conformal transformations like $C$-conformal transformations have been studied. This transformation depends on the Cartan tensor (which is invariant under the conformal change) and the conformal factor (which only depends on position). The Cartan tensor is not invariant under anisotropic conformal transformations and the conformal factor depends on the position and direction. This motivated us to study  $C$-anisotropic, $\overline{C}$-anisotropic, horizontal $C$-anisotropic and horizontal $\overline{C}$-anisotropic transformations.\\
Under the anisotropic conformal transformation $\overline{F}=e^{\phi}F$ defined by \eqref{the anisotropic conformal transformation}, the following points are to be highlighted:
\begin{description}
    \item[(i)] We have characterized $C$-anisotropic, $\overline{C}$-anisotropic, horizontal $C$-anisotropic and horizontal $\overline{C}$-anisotropic transformations.
\item[(ii)] We have determined the conditions under which the properties of $C$-anisotropic and horizontal $C$-anisotropic are preserved under the anisotropic conformal change \eqref{the anisotropic conformal transformation}.
\item [(iii)] Let \eqref{the anisotropic conformal transformation} be the anisotropic conformal transformation between two non-Riemannian surfaces then we have: 
\begin{description}
    \item[(a)] If the property of $C$-anisotropic conformal is preserved, then $F$ is projectively flat  if and only if  $\phi$ is a first integral of the geodesic spray $S.$
    \item[(b)]If $F$ is a flat metric and the anisotropic conformal transformation \eqref{the anisotropic conformal transformation} is $\overline{C}$-anisotropic conformal change, then the two sprays $S$ and $\overline{S}$ are projectively equivalent if and only if $\phi_{;2}$ is a first integral of the geodesic spray $S.$ 
    \item[(c)] If \eqref{the anisotropic conformal transformation} satisfies the horizontal $C$-anisotropic (or horizontal $\overline{C}$-anisotropic), the geodesic spray is invariant if and only if the conformal factor is first integral of the geodesic spray~$S$.
    \item[(d)] If \eqref{the anisotropic conformal transformation} satisfies the horizontal $\overline{C}$-anisotropic, the two geodesic sprays $S$ and $\overline{S}$ are projectively equivalent if and only if  $\phi_{;2}$ is a first integral of the geodesic spray $S$.
    \item[(e)] If the conformal factor is a function of $x$ only, the $C$-anisotropic, $\overline{C}$-anisotropic, horizontal $C$-anisotropic and horizontal $\overline{C}$-anisotropic transformations are reduced to the well-known $C$-conformal transformation. However, the vertical $C$-anisotropic conformal transformation does not reduce to $C$-conformal and it characterizes the property of $F$ to be a Riemannian metric. 
\end{description}
\end{description}
Secondly, in conformal transformation, a smooth non-constant function $\boldsymbol{\sigma}$ on $M$ that satisfies certain conditions with the $T$-tensor is called the $\boldsymbol{\sigma} T$-condition. Based on the fact that the $T$-tensor is not invariant under the anisotropic conformal transformation \eqref{the anisotropic conformal transformation} and the conformal factor depend on position and direction arguments, we discuss under what condition \eqref{the anisotropic conformal transformation} satisfies $\phi T$-condition, $\phi\overline{T}$-condition, horizontal $\phi T$-condition and horizontal $\phi\overline{T}$-condition. This is achieved through the existence of a function $\phi$ that satisfies certain conditions with the T and $\overline{T}$-tensors.
\begin{description}
    \item[(i)] We have established  the characterization of the anisotropic conformal transformation \eqref{the anisotropic conformal transformation} satisfies the $\phi T$-condition, $\phi\overline{T}$-condition, horizontal $\phi T$-condition and horizontal $\phi\overline{T}$-condition.
     \item[(ii)] Since any Landsberg metric satisfying the $\boldsymbol{\sigma} T$-condition is Berwaldian, we have found under what conditions a Landsberg metric $F$ (resp. $\overline{F}$) that satisfies $\phi T$-condition or horizontal $\phi T$-condition  (resp. $\phi\overline{T}$-condition or  and horizontal $\phi\overline{T}$-condition) is Berwaldian.
     \item[(iii)] Under a vertical $C$-anisotropic conformal change with horizontally constant $\phi_{;2}$, the Finsler metric \(\overline{F}\) is Landsbergian if and only if the transformation satisfies the horizontal $\phi\overline{T}$-condition, which is also equivalent to \(\overline{F}\) being Berwaldian or the two geodesic sprays $S$ and $\overline{S}$ are projectively equivalent.
     \item[(iv)] The property that the anisotropic conformal transformation \eqref{the anisotropic conformal transformation} satisfies the vertical $\phi T$-condition is equivalent to that it satisfies the $T$-condition. 
     \item[(v)] If the conformal factor depends on the position only, the $\phi T$-condition, $\phi\overline{T}$-condition, horizontal $\phi T$-condition and horizontal $\phi\overline{T}$-condition are simplified to the notion of  $\boldsymbol{\sigma} T$-condition. However, the $\phi T$-condition does not simplify to the $\boldsymbol{\sigma} T$-condition and is equivalent to the $T$-condition.  
\end{description}


\vspace{.5 cm}
Finally, we provide the following table which summarizes all special anisotropic conformal transformations that we studied.

\begin{center}
\textbf{Table: Equivalence of special anisotropic conformal changes  }
\end{center}
\begin{tabular}{ |p{4.7cm}|p{9cm}|  }
\hline
Special anisotropic change  & equivalent conditions\\
\hline
\hline
$C$-anisotropic  & F is Riemannian or $m^i\partial_i\phi=0$\\
\hline
$\overline{C}$-anisotropic  & $\overline{F}$ is Riemannian or $m^i\partial_i\phi-\varepsilon\phi_{;2}\ell^i\partial_i\phi=0$\\
\hline
horizontal $C$-anisotropic  & F is Riemannian or $\phi_{,2}=0$\\
\hline
horizontal $\overline{C}$-anisotropic  & $\overline{F}$ is Riemannian or $\phi_{,2}=\phi_{;2}\phi_{,1}$\\
\hline
vertical $C$-anisotropic  & F is Riemannian \\
\hline
\hline
$\phi T$-condition & F has vanishing $T$-tensor or $m^i\partial_i\phi=0$\\
\hline
$\phi\overline{T}$-condition  & $\overline{F}$ has vanishing $\overline{T}$-tensor or $m^i\partial_i\phi=0$\\
\hline
horizontal $\phi T$-condition & F has vanishing $T$-tensor or $\phi_{,2}=0$\\
\hline
horizontal $\phi\overline{T}$-condition  & $\overline{F}$ has vanishing $\overline{T}$-tensor or $\phi_{,2}=\phi_{;2}\phi_{,1}$\\
\hline
vertical $\phi$T  & F is Riemannian \\
\hline
\end{tabular}

\end{document}